\documentclass[12pt]{article}
\usepackage{amsmath,amssymb,amsthm,amscd,color}
\numberwithin{equation}{section}

\newcommand{\Z}{{\mathbb Z}}

\newcommand{\C}{{\mathbb C}}

\newcommand{\Se}{{\mathcal S}}
\newcommand{\A}{{\mathcal A}}

\DeclareMathOperator{\End}{End} 
\DeclareMathOperator{\Gal}{Gal} 
\DeclareMathOperator{\Mat}{Mat} 
 
\DeclareMathOperator{\id}{id}

\DeclareMathOperator{\lc}{lc} 
\DeclareMathOperator{\ld}{ld}

\newtheorem{lemma}{Lemma}
\newtheorem{theorem}{Theorem}
\newtheorem{proposition}[lemma]{Proposition}

\theoremstyle{definition}
\newtheorem{definition}{Definition}
\newtheorem{remark}{Remark}
\title{Finite-dimensional vertex algebra modules over fixed point differential subfields}
\author{Kenichiro Tanabe\footnote{Partially supported by JSPS
Grant-in-Aid for Scientific Research No. 20740002.}\\\\
Department of Mathematics\\
Hokkaido University\\
Kita 10, Nishi 8, Kita-Ku, Sapporo, Hokkaido, 060-0810\\
Japan\\\\
ktanabe@math.sci.hokudai.ac.jp}
\date{}
\begin{document}
\maketitle
\begin{abstract}
Let $K$ be a differential field  over $\C$ with derivation $D$,
$G$ a finite linear automorphism group over $K$ which preserves $D$,
and $K^G$ the fixed point subfield of $K$ under the action of $G$.
We show that every finite-dimensional vertex algebra 
$K^G$-module is contained in some twisted vertex algebra 
$K$-module.
\end{abstract}
\section{Introduction}
In \cite{B}, Borcherds defined the notion of vertex algebras
and showed that every
commutative ring $A$ with an arbitrary derivation $D$
has a structure of vertex algebra.
Every ring $A$-module naturally becomes
a vertex algebra $A$-module. 
However, this does not imply that
ring $A$-modules and vertex algebra $A$-modules
are the same.
In fact, a vertex algebra $\Z[z,z^{-1}]$-module 
which is not a ring $\Z[z,z^{-1}]$-module was given in \cite[Section 8]{B},
where $\Z[z,z^{-1}]$ is the ring of Laurent polynomials over $\Z$.
Moreover, in \cite{T} for the polynomial ring $\C[s]$ in one variable $s$ 
 with derivation $D$,
I obtained a necessary and sufficient condition on $D$ 
that there exist finite-dimensional vertex algebra $\C[s]$-modules 
which do not come from associative algebra $\C[s]$-modules.
Thus, in general vertex algebra $A$-modules and ring $A$-modules 
are certainly different. 

Let $K$ be a differential field with derivation $D$,
$G$ a finite linear automorphism group of $K$ which preserves $D$,
and $K^G$ the fixed point subfield of $K$ under the action of $G$.
In this paper, we study vertex algebra $K^G$-modules.
Here, let us recall the following conjecture on vertex operator algebras:
let $V$ be a vertex operator algebra and $G$ a finite automorphism group of $V$.
It is conjectured that under some conditions on $V$,
every irreducible module over the fixed point vertex operator subalgebra $V^G$
is contained in some irreducible $g$-twisted 
$V$-module for some $g\in G$ (cf.\cite{DVVV}).
The motivation for studying vertex algebra $K^G$-modules
is to investigate this conjecture for vertex algebras.
In Theorem \ref{theorem:correspondence}, I shall show that 
every finite-dimensional indecomposable vertex algebra $K^G$-module becomes 
a $g$-twisted vertex algebra $K$-module for some $g\in G$.
Namely, the conjecture holds for 
all finite-dimensional vertex algebra $K^G$-modules in a stronger sense.

This paper is organized as follows.
In Section 2 we recall 
some properties 
of vertex algebras and their modules.
In Section 3 we show that 
every finite-dimensional indecomposable vertex algebra $K^G$-module 
becomes a $g$-twisted vertex algebra $K$-module 
for some $g\in G$.
In Section 4  
we give the classification of the
finite-dimensional vertex algebra $\C(s)$-modules
where $\C(s)$ is the field of rational functions in one variable $s$.
In Section 5 for all quadratic extensions $K$ of $\C(s)$
and all finite-dimensional indecomposable vertex algebra $\C(s)$-modules $M$
obtained in Section 4,
we study twisted vertex algebra $K$-module structures over $M$.

\section{Preliminary}

We assume that the reader is familiar with the basic knowledge on
vertex algebras as presented in \cite{B,DLM1,LL}.

Throughout this paper, 
$\zeta_p$ is
a primitive $p$-th root of unity for a positive integer $p$
and $(V,Y,{\mathbf 1})$ is a vertex algebra.
Recall that $V$ is the underlying vector space, 
$Y(\cdot,x)$ is the linear map from $V$ to $(\End V)[[x,x^{-1}]]$,
and ${\mathbf 1}$ is the vacuum vector.
Let ${\mathcal D}$ be the endomorphism of $V$
defined by ${\mathcal D}v=v_{-2}{\mathbf 1}$ for $v\in V$.
 
First, we recall some results in \cite{B} for a vertex algebra constructed from a 
commutative associative algebra with a derivation.

\begin{proposition}{\rm\cite{B}}\label{proposition:comm-alg}
The following hold:
\begin{enumerate}
\item
Let $A$ be a commutative associative $\C$-algebra with identity element $1$ and 
$D$ a derivation of $A$.
For $a\in A$, define $Y(a,x)\in(\End A)[[x]]$ by
\begin{align*}
Y(a,x)b&=\sum_{i=0}^{\infty}\dfrac{1}{i!}(D^ia)bx^{i}
\end{align*}
for $b\in A$. Then, $(A,Y,1)$ is a vertex algebra.
\item
Let $(V,Y,{\mathbf 1})$ be a vertex algebra such that
$Y(u,x)\in(\End V)[[x]]$ for all $u\in V$.
Define a multiplication on $V$ by $u v=u_{-1}v$ for $u,v\in V$. 
Then, $V$ is a commutative associative $\C$-algebra with identity element ${\mathbf 1}$ and 
${\mathcal D}$ is a derivation of $V$.
\end{enumerate}
\end{proposition}
Throughout the rest of this section,
$A$ is a commutative associative $\C$-algebra with identity element $1$  
over $\C$ and $D$ a derivation of $A$.
Let $(A,Y,1)$ be the vertex algebra constructed from 
$A$ and $D$ in Proposition \ref{proposition:comm-alg}
and let $(M,Y_M)$ be a module over vertex algebra $A$.
We call $M$ a {\em vertex algebra $(A,D)$-module} 
to distinguish 
between modules over vertex algebra $A$ 
and modules over associative algebra $A$.

\begin{proposition}{\rm\cite{B}}\label{proposition:comm-module}
The following hold:
\begin{enumerate}
\item
Let $M$ be an associative algebra $A$-module.
For $a\in A$, define $Y_M(a,x)\in(\End_{\C}M)[[x]]$ by
\begin{align*}
Y(a,x)u&=\sum_{i=0}^{\infty}\dfrac{1}{i!}(D^ia)ux^{i}
\end{align*}
for $u\in M$. Then, $(M,Y_M)$ is a vertex algebra $(A,D)$-module.
\item
Let $(M,Y_M)$ be a vertex algebra $(A,D)$-module such that
$Y(a,x)\in(\End_{\C}M)[[x]]$ for all $a\in A$.
Define an action of $A$ on $M$ by $au=a_{-1}u$ for $a\in A$ and $u\in M$. 
Then, $M$ is an associative algebra $A$-module.
\end{enumerate}
\end{proposition}

\begin{remark}
\begin{enumerate}
\item
Let us consider the case of $D=0$.
Let $(M,Y_M)$ be an arbitrary vertex algebra $(A,0)$-module. 
For all $a\in A$, since $0=Y_{M}(0a,x)=dY_{M}(a,x)/dx$,
we see that $Y_{M}(a,x)$ is constant.
Thus, $M$ is an associative algebra $A$-module by 
Proposition \ref{proposition:comm-module}.
\item
Let us consider the case that $A$ is finite-dimensional.
Let $(M,Y_M)$ be an arbitrary vertex algebra $(A,D)$-module.
Suppose that
there exists $a\in A$ and $u\in M$ such that
$Y_{M}(a,x)u$ is not an element of $M[[x]]$.
Since 
$Y_{M}(D^ia,x)=d^iY_{M}(a,x)/dx^i$ for all $i\geq 0$,
we see that $\{Y_{M}(D^ia,x)u\ |\ i=0,1,\ldots\}$
is linearly independent.
This contradicts that $A$ is finite-dimensional.
Thus, $M$ is an associative algebra $A$-module by 
Proposition \ref{proposition:comm-module}.
\end{enumerate}
\end{remark}
For a $\C$-linear automorphism $g$ of $V$ of finite order $p$,
set $V^r=\{u\in V\ |\ gu=\zeta_p^{r}u\}, 0\leq r\leq p-1\}$.
We recall the definition of $g$-twisted $V$-modules.
\begin{definition}\label{def:weak-twisted}
A $g$-twisted $V$-module $M$ is a vector space equipped with a
linear map
\begin{equation*}
Y_M(\,\cdot\,,x) \colon V\ni v  \mapsto Y_M(v,x) = \sum_{i \in (1/p)\Z}
v_i x^{-i-1} \in (\End_{\C}M)[[x^{1/p},x^{-1/p}]]
\end{equation*}
which satisfies the following four conditions:

\begin{enumerate}
\item $Y_M(u,x) = \sum_{i \in r/p+\Z}u_i x^{-i-1}$ for $u \in V^r$.
\item $Y_M(u,x)w\in M((x^{1/p}))$ for $u \in V$ and $w \in M$.
\item $Y_M({\mathbf 1},x) = \id_M$.
\item For $u \in V^r$, $v \in V^{s}$, 
$m\in r/T+\Z,\ n\in s/T+\Z$, and $l\in\Z$,
\begin{align*}
&\sum_{i=0}^{\infty}\binom{m}{i}
(u_{l+i}v)_{m+n-i} \\
& = 
\sum_{i=0}^{\infty}\binom{l}{i}(-1)^i
\big(u_{l+m-i}v_{n+i}+(-1)^{l+1}v_{l+n-i}u_{m+i}\big).
\end{align*}
\end{enumerate}
\end{definition}
The following result is well known 
(cf. \cite[Proposition 4.8]{K} and \cite[Proposition 3]{T}).
\begin{proposition}\label{proposition:iff-1}
Let $g$ be an automorphism of $V$ of finite order $p$,
$M$ a vector space, and 
$Y_{M}(\cdot,x)$ a linear map from
$V$ to $(\End_{\C}M)[[x^{1/p},x^{-1/p}]]$ such that
for all $0\leq r\leq p-1$ and all $u\in V^r$,
$Y_{M}(u,x)=\sum_{i\in r/p+\Z}u_{i}x^{-i-1}$.
Then, $(M,Y_M)$ is a $g$-twisted $V$-module
if and only if the following five conditions hold:
\begin{itemize}
\item[{\rm (M1)}] For $u\in V$  and $w\in M$, $Y_M(u,x)w\in M((x^{1/p}))$.
\item[{\rm (M2)}]  $Y_M({\mathbf 1},x)= \mathrm{id}_M$.
\item[{\rm (M3)}]  For $u\in V^{r},v\in V^{s}$,  
$m\in r/p+\Z$, and $n\in s/p+\Z$
\begin{align*}
[u_{m},v_{n}]&=\sum_{i=0}^{\infty}\binom{m}{i}(u_{i}v)_{m+n-i}.
\end{align*}
\item[{\rm (M4)}]  
For $u\in V^{r}$, $v\in V^{s}$, $m\in r/p+\Z$, and $n\in s/p+\Z$ 
\begin{align*}
\sum_{i=0}^{\infty}\binom{m}{i}(u_{-1+i}v)_{m+n-i}&=
\sum_{i=0}^{\infty}(u_{-1+m-i}v_{n+i}+v_{-1+n-i}u_{m+i}).
\end{align*}
\item[{\rm (M5)}]  For $u\in V$, $Y_M({\mathcal D}u,x)=dY_M(u,x)/dx$.
\end{itemize}
\end{proposition}

Let ${\mathcal B}$ be a subset of $A$ which generate $A$ as a $\C$-algebra
and $g$ a linear automorphism of $A$ of finite order $p$.
For a $g$-twisted $A$-module $(M,Y_M)$,
we call $(M,Y_M)$ a {\em $g$-twisted vertex algebra $(A,D)$-module}.

\begin{lemma}
\label{lemma:comm-module}
Let $M$ be a vector space and $Y_{M}(\cdot,x)$ a linear map from
$A$ to $(\End_{\C}M)[[x^{1/p},x^{-1/p}]]$ such that
for all $0\leq r\leq p-1$ and all $a\in A^r$,
$Y_{M}(a,x)=\sum_{i\in r/p+\Z}a_{i}x^{-i-1}$.
Let $\A_{M}(A)$ denote the subalgebra of $\End_{\C}M$
generated by all $a_{i}$ where $a\in A$ and $i\in (1/p)\Z$.
Suppose that $M$ is a finitely generated $\A_{M}(A)$-module. 

Then, $(M,Y_M)$ is a $g$-twisted vertex algebra $(A,D)$-module
if and only if the following five conditions hold:
\begin{enumerate}
\item For $a\in {\mathcal B}$, $Y_M(a,x)\in (\End_{\C}M)((x^{1/p}))$.
\item $Y_M({\mathbf 1},x)= \id_M$.
\item For $a,b\in {\mathcal B}$ and $i,j\in (1/p)\Z$,
$a_{i}b_{j}=b_{j}a_{i}$.
\item For $a\in {\mathcal B}$ and $b\in A$, $Y_{M}(ab,x)=Y_{M}(a,x)Y_{M}(b,x)$.
\item For $a\in {\mathcal B}$, $Y_M(Da,x)=dY_M(a,x)/dx$.
\end{enumerate}
In this case, $Y_M(\cdot,x)$ is a $\C$-algebra homomorphism from $A$ to $(\End_{\C}M)((x^{1/p}))$.
\end{lemma}
\begin{proof}
We use Proposition \ref{proposition:iff-1}.
Note that for all $a,b\in A$ and all $i\geq 0$, 
we have $ab=a_{-1}b$ and $a_ib=0$.
Suppose that $(M,Y_{M})$ is a $g$-twisted vertex algebra $(A,D)$-module.
We have (3) by (M3).
Since $M$ is a finitely generated $\A_{M}(A)$-module,
we have (1) by (M1) and (3). 
We have (4) by (M4), (1), and (3). The other conditions clearly hold.

Conversely, suppose that $(M,Y_{M})$ satisfies the conditions (1)--(5).
It follows from (1),(2), and (4) that 
$Y_{M}(ab,x)=Y_{M}(a,x)Y_{M}(b,x)$ for all $a,b\in A$,
namely $Y(\cdot,x)$ is a $\C$-algebra homomorphism 
from $A$ to $(\End_{\C}M)((x^{1/p}))$.
This shows (M3) by (3) and hence (M4).
Let $a,b$ be elements of $A$ which satisfy (M5).
Then,
\begin{align*}
Y_M(D(ab),x)&=Y_M((Da)b+a(Db),x)\\
&=Y_M(Da,x)Y_M(b,x)+Y_M(a,x)Y_M(Db,x)\\
&=(\dfrac{d}{dx}Y_M(a,x))Y_M(b,x)+Y_M(a,x)\dfrac{d}{dx}Y_M(b,x)\\
&=\dfrac{d}{dx}(Y_M(a,x)Y_M(b,x))\\
&=\dfrac{d}{dx}(Y_M(ab,x)).
\end{align*}
Since $A$ is generated by ${\mathcal B}$, (M5) follows from (5).
We conclude that $(M,Y_M)$ is a $g$-twisted vertex algebra $(A,D)$-module.
\end{proof}

For a $g$-twisted vertex algebra $(A,D)$-module $(M,Y_{M})$ and 
a linear automorphism $h$ of $A$ which preserves $D$, 
define $(M,Y_{M})\circ h=(M\circ h, Y_{M\circ h})$
by $M\circ h=M$ as vector spaces and $Y_{M\circ h}(a,x) = Y_{M}(ha, x)$
for all $a\in A$.
Then, $(M,Y_{M})\circ h$ is an $h^{-1}gh$-twisted vertex algebra $(A,D)$-module. 

\section{Finite-dimensional vertex algebra modules over 
fixed point differential subfields}

Let $K$ be a differential field over $\C$ with derivation $D$
and let $G$ be a finite $\C$-linear automorphism group of $K$ of order $N$ which preserves $D$.
We fix a primitive element $\theta$ of $K$ over $K^G$ with
the minimal polynomial $P(Z)=\sum_{i=0}^{N}P_iZ^i\in K^{G}[Z]$ in $Z$.
For a finite-dimensional vertex algebra $(K^G,D)$-module $(M,Y_{M})$,
$g\in G$ of order $p$, and a linear map $\tilde{Y}(\cdot,x)$ from $K$ to $(\End_{\C}M)((x^{1/p}))$,
we call
$(M,\tilde{Y}_{M})$ 
a  {\em
$g$-twisted vertex algebra $(K,D)$-module structure over $(M,Y_{M})$}
if $(M,\tilde{Y}_{M})$ is a  $g$-twisted vertex algebra $(K,D)$-module 
and if $\tilde{Y}(\cdot,x)|_{K^G}=Y(\cdot,x)$.

In this section, we shall show that 
every finite-dimensional indecomposable vertex algebra $(K^G,D)$-module $(M,Y_{M})$
has a $g$-twisted vertex algebra $(K,D)$-module structure over $(M,Y_M)$
for some $g\in G$. 
Here we give the outline of the proof.
It follows from Lemma \ref{lemma:comm-module}
that for any $g\in G$,
to construct a $g$-twisted vertex algebra $(K,D)$-module structure over $(M,Y_{M})$
is equivalent to construct a $\C$-algebra homomorphism from $K$ to $(\End_{\C}M)((x^{1/|g|}))$ 
with some conditions.
The basic idea to construct such a $\C$-algebra homomorphism is to realize the Galois extension $K/K^G$
in $(\End_{\C}M)((x^{1/N}))$ where 
we identify $K^G$ with its image 
under the homomorphism 
$Y_{M}(\cdot,x)\colon K^{G}\rightarrow (\End_{\C}M)((x))$.
To do this, we first define a $\C$-algebra homomorphism
$\psi[K^G,(M,Y_M)]\colon K^G\rightarrow \C((x))$ for $(M,Y_M)$ and denote by $Q$ 
the image of $\psi[K^G,(M,Y_M)]$. 
It is well known that any finite extension of $\C((x))$ is $\C((x^{1/j}))$ for some
positive integer $j$ and $\Omega=\cup_{j=1}^{\infty}\C((x^{1/j}))$ is the algebraic closure
of $\C((x))$ (cf. \cite[Corollary 13.15]{E}). 
These results enable us to construct the extension $\hat{K}$ of $Q$ in $\Omega$
such that $\hat{K}\cong K$ as $\C$-algebras. 
Using $\theta,P(Z)$, and $\hat{K}$, we can construct the desired extension corresponding to $K$
in $(\End_{\C}M)((x^{1/N}))$.

We introduce some notation.
Let $R$ be a commutative ring and let $\Mat_n(R)$ denote 
the set of all $n\times n$ matrices with entries in $R$. 
Let $E_n$ denote the $n\times n$ identity matrix
and let $E_{ij}$ denote the matrix whose $(i,j)$ entry is $1$ and all other entries are $0$.
Define $\Delta_{k}(R)=\{(x_{ij})\in\Mat_n(R)\ |\ x_{ij}=0\mbox{ if $i+k\neq j$}\}$ for $0\leq k\leq n$.  
Then, for $a\in \Delta_k(R)$ and $b\in \Delta_{l}(R)$, we have $ab\in \Delta_{k+l}(R)$.
For $X=(x_{ij})\in \Mat_n(R)$ and $k=0,\ldots,n-1$, 
define the matrix $X^{(k)}=\sum_{i=1}^{n}x_{i,i+k}E_{i,i+k}\in \Delta_k(R)$.
For an upper triangular matrix $X$, we see that $X=\sum_{k=0}^{n-1}X^{(k)}$
and that the diagonal part of $X$ is $X^{(0)}$.
For a positive integer $m$ and 
$H=\sum_{i\in(1/m)\Z}H_{(i)}x^{i}\in (\Mat_{n}(R))[[x^{1/m},x^{-1/m}]]$ with $H_{(i)}\in \Mat_{n}(R)$,
$H^{(k)}$ denotes $\sum_{i\in(1/m)\Z}H_{(i)}^{(k)}x^{i}\in \Delta_{k}(R)[[x^{1/m},x^{-1/m}]]$ for $k=0,\ldots,n-1$.
For an $n$-dimensional vector space $M$ over $\C$,
we sometimes identify $\End_{\C}M$ with $\Mat_{n}(\C)$ and 
$(\End_{\C}M)[[x^{1/m},x^{-1/m}]]$ with $(\Mat_{n}(\C))[[x^{1/m},x^{-1/m}]]$ by fixing a basis of $M$
and use these symbols 
in the proofs of Theorem \ref{theorem:correspondence} and Theorem \ref{theorem:untwist}.

Let $A$ be a commutative associative $\C$-algebra,
$D$ a derivation of $A$, $g$ a $\C$-linear automorphism of $A$ of finite order $p$.
For a vector space $W$ over $\C$ and  a linear map $Y_{W}(\cdot,x)$ from
$A$ to $(\End W)[[x^{1/p},x^{-1/p}]]$,
${\A }_{W}(A)$ denotes the subalgebra of $\End W$
generated by all coefficients of
$Y_{W}(a,x)$ where $a$ ranges over all elements of $A$.
Let $M$ be a finite-dimensional $g$-twisted vertex algebra $(A,D)$-module. 
Then, ${\A }_{M}(A)$ is a commutative $\C$-algebra by Lemma \ref{lemma:comm-module} 
and $M$ is a finite-dimensional ${\A }_{M}(A)$-module. 
Let ${\mathcal J}_{M}(A)$ denote the Jacobson radical of ${\A }_{M}(A)$.
Since ${\A }_{M}(A)$ is a finite-dimensional commutative $\C$-algebra, 
the Wedderburn--Malcev theorem (cf.\cite[Section 11.6]{P}) says that 
$\A_{M}(A)
=\oplus_{i=1}^{m}\C e_i\oplus {\mathcal J}_{M}(A)$
and hence
\begin{align*}
{\A }_{M}(A)((x^{1/p}))
&=\oplus_{i=1}^{m}\C((x^{1/p})) e_i\oplus {\mathcal J}_{M}(A)((x^{1/p}))
\end{align*} 
where $e_1,\ldots,e_m$ are  primitive orthogonal idempotents of ${\A }_{M}(A)$.
In the case of $m=1$, which is equivalent to 
${\A }_{M}(A)$ being a local $\C$-algebra, we shall often identify the subalgebra $\C((x^{1/p}))\id$ of
${\mathcal A}_{M}(K)((x^{1/p}))$ with $\C((x^{1/p}))$.
For $X\in {\A }_{M}(A)$, $X^{[0]}$ denotes
the image of $X$ under the projection 
\begin{align}
\label{eq:proj}
{\A }_{M}(A)&=\oplus_{i=1}^{m}\C e_i\oplus {\mathcal J}_{M}(A)\nonumber\\
&\rightarrow\oplus_{i=1}^{m}\C e_i\cong \C^{m}
\end{align}
and is called the {\it semisimple part} of $X$.
For $H=\sum_{i\in(1/p)\Z}H_{(i)}x^{i}\in {\A }_{M}(A)((x^{1/p}))$ with $H_{(i)}\in {\A }_{M}(A)$,
$H^{[0]}$ denotes 
$\sum_{i\in(1/p)\Z}H_{(i)}^{[0]}x^{i}\in \C((x^{1/p}))^{\oplus m}$
and is called the {\it semisimple part} of $H$.
We denote by $\psi[{A,(M,Y_M)}]$ the $\C$-algebra homomorphism $Y_{M}(\cdot,x)^{[0]}$ 
from $A$ to $\C((x^{1/p}))^{\oplus m}$.
Note that ${\mathcal J}_{M}(A)^{n}((x^{1/p}))=0$ where $n=\dim_{\C}M$.
If $A$ is a field, then $\psi[{A,(M,Y_M)}]$ is an injective homomorphism 
from $A$ to $\C((x^{1/p}))^{\oplus m}$.
Since ${\A }_{M}(A)$ is commutative,
we sometimes identify $\End_{\C}M$ with $\Mat_{n}(\C)$
by fixing a basis of $M$ so that the representation  matrix of each element of ${\A }_{M}(A)$ with respect to the basis
is an upper triangular matrix. Under this identification,
for $H\in {\A }_{M}(A)((x^{1/p}))\subset (\End_{\C}M)((x^{1/p}))$ we see that $H^{[0]}=H^{(0)}$,
which is the diagonal part of $H$ defined above.

Let $M$ be a finite-dimensional indecomposable vertex algebra $(K^G,D)$-module.
Since ${\A }_{M}(K^G)$ is a commutative algebra, ${\A }_{M}(K^G)$ is a local algebra and hence
${\A }_{M}(K^G)=\C\id \oplus {\mathcal J}_{M}(K^G)$.
Let $(M,\tilde{Y}_{M})$ be 
a $g$-twisted vertex algebra $(K,D)$-module structure over $(M,Y_M)$.
Since ${\A }_{M}(K^G)$ 
is a subalgebra of ${\A }_{M}(K)$, 
$M$ is an indecomposable ${\A }_{M}(K)$-module.
Therefore, 
${\A }_{M}(K)$ is local since 
${\A }_{M}(K)$ is commutative.
Thus, 
${\A }_{M}(K)=\C\id\oplus {\mathcal J}_{M}(K)$
and hence
\begin{align}
\label{eq:restriction}
\psi[{K,(M,\tilde{Y}_{M})}]|_{K^{G}}&=\psi[{K^G,(M,Y_{M})}].
\end{align}
Symbol $Q$ denotes the image of the homomorphism $\psi[{K^G,(M,Y_{M})}] : K^G\rightarrow \C((x))$.
For a polynomial $F(Z)\in K^{G}[Z]$,
$\hat{F}(Z)$ denotes
the image of $F(Z)$ under the map $Y_{M}(\cdot,x) : 
K^{G}[Z]\rightarrow ({\mathcal A}_{M}(K^G)((x)))[Z]$
and $\hat{F}^{[0]}(Z)$ denotes
the image of $F(Z)$ under the map $\psi[K^G,(M,Y_{M})] :
K^{G}[Z]\rightarrow Q[Z]\subset \C((x))[Z]$.
We write 
\begin{align*}
\hat{F}(Z)
&=\sum_{i\geq 0}\hat{F}_i(x)Z^i, \hat{F}_i(x)\in
{\mathcal A}_{M}(K^G)((x))
\mbox{\quad and }\\
\hat{F}^{[0]}(Z)&=\sum_{i\geq 0}\hat{F}_{i}(x)^{[0]}Z^i,\ 
\hat{F}_{i}(x)^{[0]}\in Q.
\end{align*}
Since $P(Z)$ is an irreducible polynomial over $K^G$,
so is $\hat{P}^{[0]}(Z)$ over $Q$.

Now we state our main theorem.

\begin{theorem}\label{theorem:correspondence}
Let $K$ be a differential field over $\C$ with derivation $D$,
$G$ a finite linear automorphism group of $K$ which preserves $D$,
and $(M,Y_{M})$ a non-zero finite-dimensional indecomposable vertex algebra $(K^G,D)$-module.
Then, we have the following results:
\begin{enumerate}

\item
$M$ has a $g$-twisted vertex algebra $(K,D)$-module structure over $(M,Y_M)$ for some $g\in G$.

\item 
Let $g$ be an element of $G$ and
let $(M,\tilde{Y}^{1}_{M}), (M,\tilde{Y}^{2}_{M})$ be two
$g$-twisted vertex algebra $(K,D)$-module structures over $(M,Y_M)$
such that $\psi[{K,(M,\tilde{Y}^{1}_{M})}]=\psi[{K,(M,\tilde{Y}^{2}_{M})}]$.
Then, $(M,\tilde{Y}^{1}_{M})\cong(M,\tilde{Y}^{2}_{M})$ as 
$g$-twisted vertex algebra $(K,D)$-modules.

\item Let $g$ be an element of $G$ and
let $(M,\tilde{Y}_{M})$  be a 
$g$-twisted vertex algebra $(K,D)$-module structure over $(M,Y_M)$.
Then, $\tilde{Y}_{M}\circ h, h\in G,$ are  all distinct 
homomorphisms from $K$ to $(\End_{\C}M)((x^{1/|g|}))$. 

\item
For $k=1,2$, let $g_k\in G$ of order $p_k$
and let $(M,\tilde{Y}^{k}_{M})$  be a
$g_k$-twisted vertex algebra $(K,D)$-module structure over $(M,Y_M)$.
Then, $(M,\tilde{Y}^{1}_{M})\circ h\cong (M,\tilde{Y}^{2}_{M})$ for some $h\in G$.
\end{enumerate}
\end{theorem}
\begin{proof}
Set $n=\dim_{\C}M$ and $|G|=N$.
Let the notation be as above.
Since $K^G$ is a field, 
$\psi[{K^G,(M,Y_{M})}]$ is an injective homomorphism from $K^G$ to the field $\C((x))$
and hence $K^G\cong Q$ as $\C$-algebras.
It is well known that any finite extension of $\C((x))$ is $\C((x^{1/j}))$ for some
positive integer $j$ and $\Omega=\cup_{j=1}^{\infty}\C((x^{1/j}))$ is the algebraic closure
of $\C((x))$ (cf. \cite[Corollary 13.15]{E}). 
The field $\C((x^{1/j}))$ becomes a Galois extension of $\C((x))$ whose
Galois group is the cyclic group
generated by the automorphism sending $x^{1/j}$ to $\zeta_jx^{1/j}$. 
Let $\hat{K}$ denote 
the splitting field  for $\hat{P}^{[0]}(Z)\in Q[Z]$ in $\Omega$
and let $\Se$ denote the set of all $\C$-algebra isomorphisms from $K$ to $\hat{K}$
whose restrictions to ${K^G}$ are equal to $\psi[K^G,(M,Y_M)]$.
Since $\theta$ is a primitive element of $K$ over $K^G$,
any map of ${\mathcal S}$ is uniquely determined by the image of $\theta$ under that map.
Since $K$ is the splitting field for $P(Z)\in K^G[Z]$, 
Galois theory implies that $|\Se|=N$ and
for any two $\phi_1, \phi_2\in \Se$ 
there exists $h\in G$ such that $\phi_1(ha)=\phi_2(a)$ for all $a\in K$.

Let $g$ be an element of $G$ and 
let $(M,\tilde{Y}_{M})$ be a $g$-twisted vertex algebra $(K,D)$-module structure over $(M,Y_M)$.
By \eqref{eq:restriction}, for all root $a\in K$ of the polynomial $P(Z)$ we have
\begin{align*}
0&=\psi[K,(M,\tilde{Y}_{M})](P(a))=\hat{P}^{[0]}(\psi[K,(M,\tilde{Y}_{M})](a)).
\end{align*}
Since $\hat{K}$ is the splitting field for $\hat{P}^{[0]}(Z)\in Q[Z]$ in $\Omega$,
the image of $K$ under the map $\psi[K,(M,\tilde{Y}_{M})]$
is equal to $\hat{K}$
and hence $\psi[K,(M,\tilde{Y}_{M})]$
is an element of $\Se$.

We identify $G$ with $\Gal(\hat{K}/Q)$ throughout the argument below.

\begin{enumerate}
\item[(1)]
Let $\phi$ be an arbitrary element of $\Se$.
We use Lemma \ref{lemma:comm-module} by taking ${\mathcal B}=K^{G}\cup\{\theta\}$.
We shall construct a $\C$-algebra homomorphism $\tilde{Y}_M(\cdot,x)$ from $K\cong K^{G}[Z]/(P(Z))$ to 
$(\End_{\C}M)((x^{1/N}))$.
We first find $g\in G$ such that for all $b\in K$ with $gb=\zeta_{p}^ib,i\in\Z$,
$\phi(b)$ is an element of $x^{-i/p}\C((x))$.  
Since $K$ is a finite extension of $K^G$ and 
$Q$ is a subfield of $\C((x))$,
$Q$ is a subfield of $\hat{K}\cap \C((x))$ and $\hat{K}\C((x))=\C((x^{1/p}))$ for some positive integer $p$. 
The isomorphism 
\begin{align}
\label{eq:galois}
\Gal(\C((x^{1/p}))/\C((x)))\ni\sigma\mapsto \sigma|_{\hat{K}}\in \Gal(\hat{K}/\hat{K}\cap\C((x)))
\end{align}
implies that $\Gal(\hat{K}/\hat{K}\cap\C((x)))$ is the cyclic group of order $p$.
Let $g\in G$ be a generator of $\Gal(\hat{K}/\hat{K}\cap\C((x)))$ which 
is the homomorphic image of the element in $\Gal(\C((x^{1/p}))/\C((x)))$ sending 
$x^{1/p}$ to $\zeta_{p}x^{1/p}$ under the isomorphism \eqref{eq:galois}.
We have the eigenspace decomposition $K=\oplus_{j=0}^{p-1}K^{(g;j)}$ for $g$ 
where $K^{(g;j)}=\{a\in K\ |\ ga=\zeta_{p}^{j}a\}$.
For $a\in K^{(g;i)}$ and $b\in K^{(g;j)}, 1\leq i,j\leq p-1$,
taking two integers $i_0,j_0$ such that $i_0i+j_0j=(i,j)$,
we have $g(a^{i_0}b^{j_0})=\zeta_{p}^{(i,j)}a^{i_0}b^{j_0}$
where $(i,j)$ is the greatest common divisor of $i$ and $j$.
This implies that we can take a nonzero element $a$ of $K^{(g;1)}$ since the order of $g$ is equal to $p$.
We fix a nonzero element $a$ of $K^{(g;1)}$.
It follows from $\phi(a^{p})\in \hat{K}^{\langle g\rangle}=\hat{K}\cap\C((x))$ that
$\phi(a)$ is a root of the polynomial 
$Z^p-\phi(a^p)\in \C((x))[Z]$.
Thus, $\phi(a)$ is an element of $x^{-r/p}\C((x))$ for some integer $r$.
For all $i=1,\ldots,p-1$, it follows by $a^i\not\in K^{\langle g\rangle}$ 
that $\phi(a^i)\not\in \hat{K}^{\langle g\rangle}=\hat{K}\cap\C((x))$. 
Therefore, we have $(r,p)=1$.
Taking two integers $\gamma,\delta$ such that $\gamma r+\delta p=1$ and
replacing $a$ by $a^{\gamma}$,
we have $g a=\zeta_{p}^{\gamma }a$ and  $\phi(a)\in x^{-1/p}\C((x))$.
Since $(\gamma ,p)=1$, by replacing $g$ by a suitable power of $g$,
we have $g a=\zeta_{p}a$ and  $\phi(a)\in x^{-1/p}\C((x))$.
For all $b\in K$ with $gb=\zeta_{p}^{i}b$,
since $a^{-i}b$ is an element of $K^{\langle g\rangle}=\hat{K}\cap\C((x))$, we have $\phi(b)$
is an element of $x^{-i/p}\C((x))$.

Set $T(x)^{[0]}=\phi(\theta)\in \hat{K}$, which 
is a root of $\hat{P}^{[0]}(Z)\in Q[Z]$.
Since 
$\hat{P}^{[0]}(Z)$ is the image of $P(Z)$ under the map $\psi[K^G,(M,Y_{M})]$,
it is a separable polynomial and hence
$(d\hat{P}^{[0]}/dZ)(T(x)^{[0]})$ is not zero.
We set 
\begin{align*}
\hat{P}_i(x)^{[1]}&=\hat{P}_{i}(x)-\hat{P}_{i}(x)^{[0]}\in {\mathcal J}_{M}(K^G)((x))
\mbox{\quad and }\\
\hat{P}_{i}(x)^{[k]}&=0\in {\mathcal J}_{M}(K^G)^{k}((x))
\end{align*}
for all $i=0,\ldots, N$ and $k=2,3,\ldots$ for convenience.
For $k=1,2,\ldots,n-1$ we inductively define $T(x)^{[k]}\in {\mathcal J}_{M}(K^G)^k((x^{1/p}))$ 
by
\begin{align}
T(x)^{[k]}
&=-(\dfrac{d\hat{P}^{[0]}}{dZ}(T(x)^{[0]}))^{-1}\nonumber\\
&\quad\times
\sum_{i=0}^{N}\sum_{j_0=0}^{k}\sum_{
\begin{subarray}{l}
0\leq j_1,\ldots,j_i<k\\
j_0+j_1+\cdots+j_i=k\end{subarray}}
\hat{P}_i(x)^{[j_0]}T(x)^{[j_1]}\cdots T(x)^{[j_i]}\label{eqn:Tk}.
\end{align}
Set $T(x)=\sum_{k=0}^{n-1}T(x)^{[k]}\in{\A }_{M}(K^G)((x^{1/p}))$. 
It follows from the definition of $T(x)$ that $T(x)^{[0]}$
is exactly the semisimple part of $T(x)$ defined arter \eqref{eq:proj}.
Since ${\mathcal J}_{M}(K^G)^{n}((x))=0$, we have
\begin{align}
\label{eq:Tx}
\hat{P}(T(x))&
=\sum_{i=0}^{N}
\hat{P}_i(x)(T(x))^i\nonumber\\
&=\sum_{i=0}^{N}(\sum_{j=0}^{n-1}
\hat{P}_i(x)^{[j]})(\sum_{k=0}^{n-1}T(x)^{[k]})^i\nonumber\\
&=\sum_{k=0}^{n-1}\sum_{i=0}^{N}\sum_{
\begin{subarray}{l}
0\leq j_0,j_1,\ldots,j_i\leq n\\
j_0+j_1+\cdots+j_i=k\end{subarray}}
\hat{P}_i(x)^{[j_0]}T(x)^{[j_1]}\cdots T(x)^{[j_i]}\nonumber\\
&=
\hat{P}^{[0]}(T(x)^{[0]})\nonumber\\
&\quad{}+
\sum_{k=1}^{n-1}\sum_{i=0}^{N}\sum_{
\begin{subarray}{l}
0\leq j_0,j_1,\ldots,j_i\leq k\\
j_0+j_1+\cdots+j_i=k\end{subarray}}
\hat{P}_i(x)^{[j_0]}T(x)^{[j_1]}\cdots T(x)^{[j_i]}\nonumber\\
&=0+
\sum_{k=1}^{n-1}
\Big(T(x)^{[k]}\dfrac{d\hat{P}^{[0]}}{dZ}(T(x)^{[0]})\nonumber\\
&\quad{}+
\sum_{i=0}^{N}\sum_{j_0=0}^{k}\sum_{
\begin{subarray}{l}
0\leq j_1,\ldots,j_i<k\\
j_0+j_1+\cdots+j_i=k\end{subarray}}
\hat{P}_i(x)^{[j_0]}T(x)^{[j_1]}\cdots T(x)^{[j_i]}\Big)\nonumber\\
&=0.
\end{align}
Thus, we have an injective homomorphism $\tilde{Y}_{M}(\cdot,x)$ from 
$K\cong K^{G}[Z]/(P(Z))$ to ${\A }_{M}(K^G)((x^{1/p}))$ sending $\theta$ to $T(x)$.
Since $T(x)$ is an element of ${\A }_{M}(K^G)((x^{1/p}))$, the $\C$-algebra
${\A }_{M}(K)$ for $\tilde{Y}_{M}(\cdot,x)$ is 
equal to ${\A }_{M}(K^G)$.
In particular, ${\A }_{M}(K)$ is a commutative $\C$-algebra. 
Moreover, it follows from $\phi(\theta)=T(x)^{[0]}$ that $\psi[K,(\tilde{Y}_{M},M)]=\phi$.

Let $b$ be an element of $K$ with the minimal polynomial $R(Z)\in K^{G}[Z]$ over $K^G$
and let $B(x)$ denotes $\tilde{Y}_{M}(b,x)$.
We write $R(Z)=\sum_{i=0}^{m}R_{i}Z^i\in K^{G}[Z], R_i\in K^G$.
We shall show that if $B(x)^{[0]}$ is an element of $\C((x))\id$, then
$B(x)$ is an element of $(\End_{\C}M)((x))$.
By definition, $\hat{R}(B(x))=0$ in $(\End_{\C}M)((x^{1/p}))$ and 
$\hat{R}^{[0]}(B(x)^{[0]})=0$ in $\Omega\id$.
We identify $\End_{\C}M$ with $\Mat_{n}(\C)$
by fixing  a basis of $M$ so that the representation  matrix of each element of ${\A }_{M}(K^G)$
with respect to the basis
is an upper triangular matrix. 
We use the expansion $B(x)=\sum_{k=0}^{n-1}B(x)^{(k)},
B(x)^{(k)}\in \Delta_{k}(\C((x^{1/p})))$.
We recall $B(x)^{[0]}=B(x)^{(0)}$ and
$\hat{R}_i(x)^{[0]}=\hat{R}_i(x)^{(0)}, i=0,\ldots,m$.
It may be possible that $B(x)^{(k)}$ and $\hat{R}_i(x)^{(k)}$ are not elements of $\A_{M}(K^G)((x^{1/p}))$ for $k=1,\ldots,n-1$.
Since $B(x)^{(0)}$ and $\hat{R}_i(x)^{(0)}$ are elements of 
$\C((x^{1/p}))E_n$, they commute any element of $(\End_{\C}M)((x^{1/p}))$.
We denote the diagonal part of $\hat{R}(Z)$ by $\hat{R}^{(0)}(Z)$, namely,
$\hat{R}^{(0)}(Z)=\sum_{i=0}^{m}\hat{R}_{i}(x)^{(0)}Z^i$.
Under the identification of $\End_{\C}M$ with $\Mat_{n}(\C)$ above,
$\hat{R}^{(0)}(Z)$ is equal to $\hat{R}^{[0]}(Z)$.
The same computation as \eqref{eq:Tx} shows
\begin{align}
\label{eq:Bx}
0&=\hat{R}(B(x))
=\sum_{i=0}^{m}
\hat{R}_i(x)(B(x))^i\nonumber\\
&=\sum_{i=0}^{m}(\sum_{j=0}^{n-1}
\hat{R}_i(x)^{(j)})(\sum_{k=0}^{n-1}B(x)^{(k)})^i\nonumber\\
&=\sum_{k=0}^{n-1}\sum_{i=0}^{m}\sum_{
\begin{subarray}{l}
0\leq j_0,j_1,\ldots,j_i\leq n\\
j_0+j_1+\cdots+j_i=k\end{subarray}}
\hat{R}_i(x)^{(j_0)}B(x)^{(j_1)}\cdots B(x)^{(j_i)}\nonumber\\
&=
\hat{R}^{(0)}(B(x)^{(0)})\nonumber\\
&\quad{}+
\sum_{k=1}^{n-1}\sum_{i=0}^{m}\sum_{
\begin{subarray}{l}
0\leq j_0,j_1,\ldots,j_i\leq k\\
j_0+j_1+\cdots+j_i=k\end{subarray}}
\hat{R}_i(x)^{(j_0)}B(x)^{(j_1)}\cdots B(x)^{(j_i)}\nonumber\\
&=\sum_{k=1}^{n-1}
\Big(B(x)^{(k)}\dfrac{d\hat{R}^{(0)}}{dZ}(B(x)^{(0)})\nonumber\\
&\quad{}+
\sum_{i=0}^{m}\sum_{j_0=0}^{k}\sum_{
\begin{subarray}{l}
0\leq j_1,\ldots,j_i<k\\
j_0+j_1+\cdots+j_i=k\end{subarray}}
\hat{R}_i(x)^{(j_0)}B(x)^{(j_1)}\cdots B(x)^{(j_i)}\Big).
\end{align}
Since 
$\hat{R}^{(0)}(Z)\in Q[Z]$ is the image of $R(Z)$ under the map $\psi[K^G,(M,Y_{M})]$,
$\hat{R}^{(0)}(Z)$ is a separable polynomial and hence
$(d\hat{R}^{(0)}/dZ)(B(x)^{(0)})$ is not zero.
Thus, we have
\begin{align*}
B(x)^{(k)}
&=-(\dfrac{d\hat{R}^{(0)}}{dZ}(B(x)^{(0)}))^{-1}\\
&\quad\times
\sum_{i=0}^{m}\sum_{j_0=0}^{k}\sum_{
\begin{subarray}{l}
0\leq j_1,\ldots,j_i<k\\
j_0+j_1+\cdots+j_i=k\end{subarray}}
\hat{R}_i(x)^{(j_0)}B(x)^{(j_1)}\cdots B(x)^{(j_i)}.
\end{align*}
Since $B(x)^{(0)}=\phi(b)\in\C((x))E_n$ and 
$\hat{R}_i(x)\in (\End_{\C}M)((x))$, it follows by induction on $k$ that
$B(x)^{(k)}$ is an element of $(\End_{\C}M)((x))$.
We conclude that $B(x)$ is an element of $(\End_{\C}M)((x))$.

For all $b\in K$ with $gb=\zeta_{p}^{i}b$,
we shall show that
$\tilde{Y}_{M}(b,x)$ is an element of $x^{-i/p}(\End_{\C}M)((x))$.
Set $B(x)=\tilde{Y}_{M}(b,x)$ and $C(x)=B(x)^{p}$.
It follows from $b^p\in K^{\langle g\rangle}$ that
$C(x)^{[0]}=(B(x)^{[0]})^{p}$ is an element of $\hat{K}^{\langle g\rangle}
=\hat{K}\cap \C((x))$. 
The argument above shows $C(x)$ is an element of $(\End_{\C}M)((x))$.
We identify $\End_{\C}M$ with $\Mat_{n}(\C)$
by fixing  a basis of $M$ so that the representation  matrix of each element of ${\A }_{M}(K^G)$
with respect to the basis
is an upper triangular matrix. 
We use the expansion $B(x)=\sum_{k=0}^{n-1}B(x)^{(k)},
B(x)^{(k)}\in \Delta_{k}(\C((x^{1/p})))$.
We have already seen in the first part of the proof of (1) that $B(x)^{(0)}=\phi(b)$ is an element of $x^{-i/p}(\End_{\C}M)((x))$.
By $C(x)=B(x)^p$, the same computation as \eqref{eq:Bx}
shows
\begin{align*}
B(x)^{(k)}
&=-p^{-1}(B(x)^{(0)})^{-p+1}\nonumber\\
&\quad\times (C(x)^{(k)}+
\sum_{
\begin{subarray}{l}
0\leq j_1,\ldots,j_{p}<k\\
j_1+\cdots+j_{p}=k\end{subarray}}
B(x)^{(j_1)}\cdots B(x)^{(j_p)}).
\end{align*} 
for all $k=1,\ldots,n-1$.
It follows by induction on $k$ that
$B(x)^{(k)}$ is an element of $x^{-i/p}(\End_{\C}M)((x))$. 
We conclude that $B(x)$ is an element of $x^{-i/p}(\End_{\C}M)((x))$.

We shall show $\tilde{Y}_{M}(D\theta,x)=d\tilde{Y}_{M}(\theta,x)/dx$.
It follows from $P(\theta)=0$
that 
\begin{align*}
0&=D(P(\theta))=\sum_{i=0}^{N}(DP_i)\theta^i+\big(\dfrac{dP}{dZ}(\theta)\big)(D\theta).
\end{align*}
and hence
\begin{align}
\label{eq:d-1}
0&=\sum_{i=0}^{N}\tilde{Y}_{M}(DP_i,x)\tilde{Y}_{M}(\theta,x)^i+
\tilde{Y}_{M}(\dfrac{dP}{dZ}(\theta),x)\tilde{Y}_{M}(D\theta,x).
\end{align}
We also have
\begin{align}
\label{eq:d-2}
0&=\dfrac{d}{dx}\tilde{Y}_{M}(P(\theta),x)\nonumber\\
&=\dfrac{d}{dx}\sum_{i=0}^{N}\tilde{Y}_{M}(P_i,x)\tilde{Y}_{M}(\theta,x)^i\nonumber\\
&=\sum_{i=0}^{N}\dfrac{d\tilde{Y}_{M}(P_i,x)}{dx}\tilde{Y}_{M}(\theta,x)^i+
\sum_{i=0}^{N}\tilde{Y}_{M}(P_i,x)i\tilde{Y}_{M}(\theta,x)^{i-1}\dfrac{\tilde{Y}_{M}(\theta,x)}{dx}
\nonumber\\
&=\sum_{i=0}^{N}\dfrac{dY_{M}(P_i,x)}{dx}\tilde{Y}_{M}(\theta,x)^i+
\tilde{Y}_{M}(\dfrac{dP}{dZ}(\theta),x)\tilde{Y}_{M}\dfrac{\tilde{Y}_{M}(\theta,x)}{dx}.
\end{align}
Since $P(Z)\in K^{G}[Z]$ is a separable polynomial in $Z$,
$(dP/dZ)(\theta)$ is not zero and hence
$\tilde{Y}_{M}((dP/dZ)(\theta),x)$ is an invertible element of $(\End_{\C}M)((x^{1/p}))$.
Since $Y_{M}(DP_i,x)=d{Y}_{M}(P_i,x)/dx$ for all $i$,
we have $\tilde{Y}_{M}(D\theta,x)=d\tilde{Y}_{M}(\theta,x)/dx$
by \eqref{eq:d-1} and \eqref{eq:d-2}.

By Lemma \ref{lemma:comm-module}, we conclude that 
$(M,\tilde{Y}_{M})$ is a $g$-twisted vertex algebra $(K,D)$-module structure over 
$(M,Y_M)$.

\item[(2)]
We denote the order of $g$ by $p$
and $\psi[{K,(M,\tilde{Y}^{1}_{M})}]=\psi[{K,(M,\tilde{Y}^{2}_{M})}]$ 
by $\Psi$ for simplicity.
We recall $\Psi$ is an element of $\Se$ as mentioned before starting the proof of (1).

Let $(M,\tilde{Y}_{M})$ be a
$g$-twisted vertex algebra $(K,D)$-module structure over $(M,Y_M)$ constructed in (1) by 
taking $\Psi$ as $\phi$.
It is enough to show that $\tilde{Y}_{M}=\tilde{Y}^{1}_{M}$.
We denote ${\A }_{M}(K)$ for $(M,\tilde{Y}_{M})$
by ${\A }^{0}$ and 
${\A }_{M}(K)$ for $(M,\tilde{Y}^{1}_{M})$
by ${\A }^1$.
We have seen in (1) that ${\A }^{0}={\A }_{M}(K^G)$.
Thus, ${\A }^{0}$ is a subalgebra of ${\A }^1$.
We identify $\End_{\C}M$ with $\Mat_{n}(\C)$
by fixing a basis of $M$ so that the representation  matrix of each element of $\A^1$ with respect to the basis
is an upper triangular matrix. 
We denote $\tilde{Y}^{1}_{M}(\theta,x)$ by $U(x)$.
We use the expansions 
\begin{align*}
U(x)&=\sum_{k=0}^{n-1}U(x)^{(k)},\ U(x)^{(k)}\in \Delta_{k}(\C((x^{1/p})))
\mbox{ and }\\
\hat{P}_i(x)&=\sum_{k=0}^{n-1}\hat{P}_i(x)^{(k)},
\hat{P}_i(x)^{(k)}\in \Delta_{k}(\C((x^{1/p}))).
\end{align*}
We denote the diagonal part of $\hat{P}(Z)$ by $\hat{P}^{(0)}(Z)$, namely,
$\hat{P}^{(0)}(Z)=\sum_{i=0}^{N}\hat{P}_{i}(x)^{(0)}Z^i$.
Under the identification of $\End_{\C}M$ with $\Mat_{n}(\C)$ above,
$\hat{P}^{(0)}(Z)$ is equal to $\hat{P}^{[0]}(Z)$.
Note that 
we do not assume $U(x)\in {\A }_{M}(K^G)((x^{1/p}))$.
We recall $\Psi(\theta)=U(x)^{(0)}$ by definition.
We have $\hat{P}^{(0)}(U(x)^{(0)})=\Psi(P(\theta))=0$ and 
$(d\hat{P}^{(0)}/dZ)(U(x)^{(0)})$ is not zero
since $\hat{P}^{[0]}(Z)$ is a separable polynomial.
The same computation as \eqref{eq:Bx}
shows
\begin{align*}
0&=\hat{P}(U(x))\\
&=\hat{P}^{(0)}(U(x)^{(0)})\\
&\quad{}+
\sum_{k=1}^{n-1}\sum_{i=0}^{N}\sum_{
\begin{subarray}{l}
0\leq j_0,j_1,\ldots,j_i\leq k\\
j_0+j_1+\cdots+j_i=k\end{subarray}}
\hat{P}_i(x)^{(j_0)}U(x)^{(j_1)}\cdots U(x)^{(j_i)}\\
&=\sum_{k=1}^{n-1}\Big(
U(x)^{(k)}\dfrac{d\hat{P}^{(0)}}{dZ}(U(x)^{(0)})\\
&\quad{}+
\sum_{i=0}^{N}\sum_{j_0=0}^{k}\sum_{
\begin{subarray}{l}
0\leq j_1,\ldots,j_i<k\\
j_0+j_1+\cdots+j_i=k\end{subarray}}
\hat{P}_i(x)^{(j_0)}U(x)^{(j_1)}\cdots U(x)^{(j_i)}\Big)
\end{align*}
and hence
\begin{align*}
U(x)^{(k)}
&=-(\dfrac{d\hat{P}^{[0]}}{dZ}(\Psi(\theta))^{-1}\\
&\quad\times
\sum_{i=0}^{N}\sum_{j_0=0}^{k}\sum_{
\begin{subarray}{l}
0\leq j_1,\ldots,j_i<k\\
j_0+j_1+\cdots+j_i=k\end{subarray}}
\hat{P}_i(x)^{(j_0)}U(x)^{(j_1)}\cdots U(x)^{(j_i)}.
\end{align*}
for all $k=1,\ldots,n-1$.
It follows by induction on $k$ that
$U(x)=\sum_{k=0}^{n-1}U(x)^{(k)}$ is uniquely determined by $\Psi(\theta)$ and $\hat{P}(Z)$. 
By definition of $\Psi$, we have
$\tilde{Y}_{M}(\theta,x)=U(x)=\tilde{Y}^{1}_{M}(\theta,x)$
and hence $\tilde{Y}_{M}=\tilde{Y}^{1}_{M}$.

\item[(3)]
Let $h\in G$ with $h\neq 1$.
Since $h^{-1}(\theta)\neq \theta$ and 
$\tilde{Y}_{M\circ h}(h^{-1}\theta,x)=\tilde{Y}_{M}(\theta,x)$,
$\tilde{Y}_{M\circ h}(\cdot,x)$ is distinct from $\tilde{Y}_{M}(\cdot ,x)$.
This implies that 
$\tilde{Y}_{M}\circ h, h\in G,$ are  all distinct 
homomorphisms from $K$ to $(\End_{\C}M)((x^{1/|g|}))$. 
\item[(4)]
For $k=1,2$, let $g_k\in G$ of order $p_k$
and let $(M,\tilde{Y}^{k}_{M})$  be a
$g_k$-twisted vertex algebra $(K,D)$-module structure over $(M,Y_M)$.
We denote $\psi[K,(M,\tilde{Y}^{k}_{M})]$ by $\psi_k$ and 
$\psi[{K^G,(M,Y_M)}]$ by $\psi$ briefly.

Since $\psi_1,\psi_2\in\Se$, there exists $h\in G$ such that
$\psi_1(ha)=\psi_2(a)$ for all $a\in K$.
By definition of $(M,\tilde{Y}^{1}_{M})\circ h$,
we have $\psi[K,(M,\tilde{Y}^{1}_{M})\circ h](a)=\psi_1(ha)=\psi_2(a)$ for all $a\in K$.
We conclude that $(M,\tilde{Y}^{1}_{M})\circ h\cong (M,\tilde{Y}^{2}_{M})$ by (2).
\end{enumerate}
\end{proof}

\section{Finite-dimensional vertex algebra $\C(s)$-modules}

Throughout the rest of this paper,
$\C(s)$ is the field of rational functions in one variable $s$. 
In this section we classify the finite-dimensional vertex algebra $\C(s)$-modules.
We use the notation introduced in Section 3.
It is easy to see that 
every non-zero derivation $D$ of $\C(s)$ can be expressed as $D=(p(s)/q(s))d/ds$, 
where $p(s)$ and $q(s)$ are non-zero coprime elements of $\C[s]$.
We write 
\begin{align*}
p(s)=\sum_{i=L_p}^{N_p}p_is^i
\mbox{\quad and\quad }
q(s)=\sum_{i=L_q}^{N_q}q_is^i
\end{align*}
where $p_{L_p},p_{N_p},q_{L_q},q_{N_q}$ are all non-zero
complex numbers.

The following lemma is a corollary of 
Lemma \ref{lemma:comm-module}.

\begin{lemma}\label{lemma:finite}
Let the notation be as above.
Let $M$ be a finite-dimensional vector space
and let $S(x)=\sum_{i\in \Z}S_{(i)}x^{i}$ be an element of $(\End_{\C}M)((x))$.
Then, there exists a vertex algebra $(\C(s),D)$-module $(M,Y_M)$ 
with $Y_M(s,x)=S(x)$ if and only if the following three conditions hold:
\begin{itemize}
\item[(i)] For all $\alpha\in \C$, $S(x)-\alpha$ is an invertible element of $(\End_{\C}M)((x))$.
\item[(ii)] For all $i,j\in \Z$, $S_{(i)}S_{(j)}=S_{(j)}S_{(i)}$.
\item[(iii)] $dS(x)/dx=p(S(x))/q(S(x))$.
\end{itemize}
In this case, for $u(s)\in\C(s)$ we have $Y_M(u(s),x)=u(S(x))$ and hence
$(M,Y_M)$ is uniquely determined by $S(x)$.
\end{lemma}
\begin{proof}
We use Lemma \ref{lemma:comm-module} by taking ${\mathcal B}=\{s\}\cup\{(s-\alpha)^{-1}\ |\ \alpha\in\C\}$.
Suppose that $(M,Y_{M})$ is a vertex algebra $(\C(s),D)$-module.
Then, $Y_{M}(\cdot,x)$ is a $\C$-algebra homomorphism 
from $\C(s)$ to $(\End_{\C}M)((x))$.
For all $\alpha\in\C$, since $s-\alpha$ is an invertible element in $\C(s)$,
so is $Y_{M}(s-\alpha,x)=S(x)-\alpha$ in $(\End_{\C}M)((x))$.
The other conditions in  (ii) and (iii) are clearly hold.

Conversely, suppose that $(M,Y_{M})$ satisfies the conditions (i)--(iii).
Since $S(x)-\alpha$ is an invertible element of $(\End_{\C}M)((x))$ for all $\alpha\in\C$,
we obtain $q(S(x))\neq 0$ and we can define
the $\C$-algebra homomorphism $Y_M\colon \C(s)\ni u(s)\mapsto u(S(x))\in (\End_{\C}M)((x))$.
Since for all $\alpha\in\C$ each coefficient of $x^j, j\in\Z$ in the expansion of $(S(x)-\alpha)^{-1}$ is 
a polynomial in $\{S_{(k)}\ |\ k\in\Z\}$,
${\A }_{M}(\C(s))$ is commutative.
For all $\alpha\in\C$ we have 
\begin{align*}
Y_{M}(D((s-\alpha)^{-1}),x)&=
Y_{M}(-(Ds)(s-\alpha)^{-2},x)\\
&=-Y_{M}(Ds,x)(Y_{M}(s,x)-\alpha)^{-2}\\
&=
\dfrac{d}{dx}(Y_{M}(s,x)-\alpha)^{-1}.
\end{align*}
We conclude that $(M,Y_M)$ is a vertex algebra $(\C(s),D)$-module.
\end{proof}

Note that there is no nontrivial finite-dimensional associative algebra $\C(s)$-module 
since $\C(s)$ is an infinite-dimensional $\C$-vector space and $\C(s)$ is a field.

Let $M$ be a finite-dimensional vector space over $\C$.
For $X\in \End_{\C}M$, $X^{[0]}$ denotes the semisimple part of $X$
and $X^{[1]}$ denotes the nilpotent part of $X$.
For $H(x)=\sum_{i\in \Z}H_{(i)}x^{i}\in (\End_{\C}M)[[x,x^{-1}]]$,
$H(x)^{[0]}$ denotes $\sum_{i\in \Z}H_{(i)}^{[0]}x^{i}$ and
$H(x)^{[1]}$ denotes $\sum_{i\in \Z}H_{(i)}^{[1]}x^{i}$.
We call $H(x)^{[0]}$ the semisimple part of $H(x)$.
For $H(x)=\sum_{i=L}^{\infty}H_{(i)}x^i\in (\End_{\C}M)((x))$ with $H_{(L)}\neq 0$, 
$\ld(H(x))$ denotes $L$ and $\lc(H(x))$ denotes $H_{(L)}$.

For a finite-dimensional indecomposable vertex algebra $\C(s)$-module $(M,Y_M)$,
we denote $Y_{M}(s,x)$ by $S(x)$.
Let $J_n$ denote the following $n\times n$ matrix:
\begin{align*}
J_n&=\begin{pmatrix}
0&1&0&\cdots&0\\
\vdots&\ddots&\ddots&\ddots&\vdots\\
\vdots&&\ddots&\ddots&0\\
\vdots&&&\ddots&1&\\
0&\cdots&\cdots&\cdots&0
\end{pmatrix}.
\end{align*}

\begin{theorem}\label{theorem:untwist}
Let the notation be as above.
Let $\alpha$ be a non-zero complex number.
We have the following results:
\begin{enumerate}
\item There exists a non-zero finite-dimensional indecomposable vertex algebra $(\C(s),D)$-module $M$
with $\ld(S(x)^{[0]})>0$ and with $\lc(S(x)^{[0]})=\alpha$
if and only if $p(0)q(0)\neq 0$ and 
$\alpha=p(0)/q(0)$.
Moreover, in this case $\ld(S(x)^{[0]})=1$ and $S(x)\in (\End_{\C}M)[[x]]$.
\item 
There exists a non-zero finite-dimensional indecomposable  vertex algebra $(\C(s),D)$-module $M$
with $\ld(S(x)^{[0]})=0$ and with $\lc(S(x)^{[0]})=\alpha$
if and only if $p(\alpha)q(\alpha)\neq 0$. Moreover, in this case $S(x)\in (\End_{\C}M)[[x]]$
and $S^{[0]}_{(1)}=p(\alpha)/q(\alpha)$.
\item There exists a non-zero finite-dimensional indecomposable  vertex algebra $(\C(s),D)$-module $M$
with $\ld(S(x)^{[0]})<0$ and with $\lc(S(x)^{[0]})=\alpha$
if and only if $\deg p=\deg q+2$ and $\alpha=-q_{N_q}/p_{N_p}$.
Moreover, in this case $\ld(S(x)^{[0]})=-1$.
\end{enumerate}
In each case, for all positive integers $n$,
there exists a unique $n$-dimensional indecomposable vertex algebra $(\C(s),D)$-module
which satisfies the conditions up to isomorphism.
\begin{proof}
\begin{enumerate}
\item
Let $n$ be a positive integer and
let $(M,Y_{M})$ be an $n$-dimensional indecomposable vertex algebra $(\C(s),D)$-module
and let $S(x)$ denotes $Y_{M}(s,x)$.
Since $M$ is indecomposable, $S(x)^{[0]}$ is an element of $\C((x))\id$.
If $S(x)^{[0]}=0$, then $S(x)$ is nilpotent. This is impossible since $\C(s)$ is a field. 
Thus, we write $S(x)^{[0]}=\sum_{i=L}^{\infty}S^{[0]}_{(i)}x^i$ where $L=\ld(S(x)^{[0]})$.
Since $D=(p(s)/q(s))d/ds$, we have
\begin{align}\label{eqn:sqp}
\dfrac{dS(x)}{dx}
&=q(S(x))^{-1}p(S(x))
\end{align}
or equivalently
\begin{align}\label{eqn:qsp}
q(S(x))\dfrac{dS(x)}{dx}&=p(S(x)).
\end{align}
Since $q(S(x))$ is an invertible element of $(\End_{\C}M)((x))$,
$q(S(x)^{[0]})$ is not zero.
Taking the semisimple part of \eqref{eqn:qsp}, we have
\begin{align}
\label{eqn:qs0p}
q(S(x)^{[0]})\dfrac{dS(x)^{[0]}}{dx}
&=p(S(x)^{[0]}).
\end{align}
Suppose that $L$ is a positive integer, namely $S(x)^{[0]}$ is an element of $x\C[[x]]$.
Then, 
$S^{[0]}_{(0)}=0$ and hence $S_{(0)}$ is a nilpotent element of $\End_{\C}M$.

We shall show $p(0)q(0)\neq 0$, 
$\lc(S(x)^{[0]})=1$, and $\lc(S(x)^{[0]})=p(0)/q(0)$.
In \eqref{eqn:qs0p}, 
the term with the lowest degree of the left-hand side is
$q_{L_q}L(S^{[0]}_{(L)})^{L_q+1}x^{L(L_q+1)-1}$
and the term with the lowest degree of the right-hand side is
$p_{L_p}(S^{[0]}_{(L)})^{L_p}x^{LL_p}$.
Comparing these terms, we have $L(L_p-L_q-1)=-1$.
Therefore, $L=1$ and $L_p=L_q$. We have
$L_p=L_q=0$ since $p(x)$ and $q(x)$ are coprime.
Thus, both $p_0, q_0$ are not zero.
Comparing the coefficients of these terms, we also have $S^{[0]}_{(1)}=p_{0}/q_{0}=p(0)/q(0)$.

We shall show that $S(x)$ is an element of $(\End_{\C}M)[[x]]$.
In order to do that, we identify $\End_{\C}M$ with $\Mat_{n}(\C)$
by fixing  a basis of $M$ so that the representation  matrix of each element of ${\A }_{M}(\C(s))$ with respect to the basis
is upper triangular matrix for a while. 
We write the expansion $S(x)=\sum_{k=0}^{n-1}S(x)^{(k)},
S(x)^{(k)}\in \Delta_{k}(\C((x)))$ as in Section 3.
We recall that $S(x)^{[0]}$ is equal to the diagonal part $S(x)^{(0)}$.
We need to expand $(dS(x)/dx)q(S(x))$ and $p(S(x))$.
By the same computation as \eqref{eq:Tx}, we have 
\begin{align*}
&\dfrac{dS(x)}{dx}q(S(x))\\
&
=\sum_{j_0=0}^{n-1}\dfrac{dS(x)^{(j_0)}}{dx}
\big(
\sum_{i=0}^{N_q}q_i
\sum_{
0\leq j_1,\ldots,j_i\leq n-1}
S(x)^{(j_1)}\cdots S(x)^{(j_i)}\big)\\
&
=\sum_{i=0}^{N_q}q_i
\sum_{
0\leq j_0,j_1,\ldots,j_i\leq n-1}
\dfrac{dS(x)^{(j_0)}}{dx}
S(x)^{(j_1)}\cdots S(x)^{(j_i)}
\\
&=
\sum_{k=0}^{n-1}
\sum_{i=0}^{N_q}q_i
\sum_{\begin{subarray}{c}
0\leq j_0,j_1,\ldots,j_i\leq k\\
j_0+j_1+\cdots+j_i=k\end{subarray}}
\dfrac{dS(x)^{(j_0)}}{dx}S(x)^{(j_1)}\cdots S(x)^{(j_i)}
\\
&=\dfrac{dS(x)^{(0)}}{dx}q(S(x)^{(0)})\\
&\quad{}+
\sum_{k=1}^{n-1}\Big(\dfrac{dS(x)^{(k)}}{dx}q(S(x)^{(0)})+
\dfrac{dS(x)^{(0)}}{dx}\dfrac{dq}{ds}(S(x)^{(0)})S(x)^{(k)}\\
&\quad{}+
\sum_{i=0}^{N_q}q_i
\sum_{\begin{subarray}{c}
0\leq j_0,j_1,\ldots,j_i<k\\
j_0+j_1+\cdots+j_i=k\end{subarray}}
\dfrac{dS(x)^{(j_0)}}{dx}S(x)^{(j_1)}\cdots S(x)^{(j_i)}\Big)
\end{align*}
and
\begin{align*}
p(S(x))&=
p(S(x)^{(0)})+\sum_{k=1}^{n-1}\big(\dfrac{dp}{ds}(S(x)^{(0)})S(x)^{(k)}\\
&\quad{}+
\sum_{i=0}^{N_p}p_i
\sum_{\begin{subarray}{c}
0\leq j_1,\ldots,j_i<k\\
j_1+\cdots+j_i=k\end{subarray}}
S(x)^{(j_1)}\cdots S(x)^{(j_i)}\big).
\end{align*}
Thus, it follows from \eqref{eqn:qsp} that 
for all $k=1,2,\ldots,n-1$
\begin{align}
\label{eqn:semi-nil}
&\dfrac{dS(x)^{(k)}}{dx}\nonumber\\
&=q(S(x)^{(0)})^{-1}
\Big(\big(-\dfrac{dS(x)^{(0)}}{dx}\dfrac{dq}{ds}(S(x)^{(0)})
+\dfrac{dp}{ds}(S(x)^{(0)})\big)S(x)^{(k)}\nonumber\\
&\quad{}-\sum_{i=0}^{N_q}q_i
\sum_{\begin{subarray}{c}
0\leq j_0,j_1,\ldots,j_i<k\\
j_0+\cdots+j_i=k\end{subarray}}
\dfrac{dS(x)^{(j_0)}}{dx}S(x)^{(j_1)}\cdots S(x)^{(j_i)}\nonumber\\
&\quad{}+
\sum_{i=0}^{N_p}p_i
\sum_{\begin{subarray}{c}
0\leq j_1,\ldots,j_i<k\\
j_1+\cdots+j_i=k\end{subarray}}
S(x)^{(j_1)}\cdots S(x)^{(j_i)}\Big).
\end{align}
Now we show that $S(x)^{(k)}$ is an element of $(\End_{\C}M)[[x]]$ by induction on $k$.
The case $k=0$ follows from $\ld(S(x)^{(0)})=1$.
For $k>0$, suppose that $\ld(S(x)^{(k)})<0$. This implies
$\ld(dS(x)^{(k)}/dx)=\ld(S(x)^{(k)})-1$.
Since $q_0\neq 0$ and $S(x)^{(0)}$ is an element of $x\C[[x]]$, $q(S(x)^{(0)})^{-1}$ is an element of $\C[[x]]$.
Thus, the lowest degree of the right-hand side of \eqref{eqn:semi-nil} is 
greater than or equal to $\ld(S(x)^{(k)})$ by the induction assumption.
This contradicts $\ld(dS(x)^{(k)}/dx)=\ld(S(x)^{(k)})-1$. 
We conclude that $S(x)$ is an element of $(\End_{\C}M)[[x]]$.

We shall show that $S(x)$ is uniquely determined by $D$ and $S_{(0)}$.
For all positive integer $m$, we expand $(dS(x)/dx)q(S(x))$ and $p(S(x))$ 
modulo $x^{m}\C[[x]]$:
\begin{align}
\label{eq:expand-1}
& \dfrac{dS(x)}{dx}q(S(x))
\equiv(\sum_{j=1}^{m}jS_{(j)}x^{j-1})
\sum_{i=0}^{N_q}q_i(\sum_{j=0}^{m-1}S_{(j)}x^j)^i\nonumber\\
&\equiv\sum_{k=0}^{m-1}
\sum_{i=0}^{N_q}q_i
\sum_{j_0=1}^{m}
\sum_{\begin{subarray}{c}
0\leq j_1,\ldots,j_i\leq m-1\\
j_0+j_1+\cdots+j_i=k+1\end{subarray}}
j_0S_{(j_0)}S_{(j_1)}\cdots S_{(j_i)}x^{k}\nonumber\\
&=
mS_{(m)}q(S_{(0)})x^{m-1}\nonumber\\
&\quad{}+\sum_{k=0}^{m-1}
\sum_{i=0}^{N_q}q_i\sum_{j_0=1}^{m-1}
\sum_{\begin{subarray}{c}
0\leq j_1,\ldots,j_i\leq m-1\\
j_0+j_1+\cdots+j_i=k+1\end{subarray}}
j_0S_{(j_0)}S_{(j_1)}\cdots S_{(j_i)}x^k \pmod{x^{m}\C[[x]]}
\end{align}
and 
\begin{align}
\label{eq:expand-2}
p(S(x))& \equiv 
\sum_{k=0}^{m-1}
\sum_{i=0}^{N_p}p_i
\sum_{\begin{subarray}{c}
0\leq j_1,\ldots,j_i\leq m-1\\
j_1+\cdots+j_i=k\end{subarray}}
S_{(j_1)}\cdots S_{(j_i)}x^k \pmod{x^{m}\C[[x]]}.
\end{align}
Comparing the coefficients of $x^{m-1}$ in \eqref{eq:expand-1} and \eqref{eq:expand-2},
it follows from \eqref{eqn:qsp} that
\begin{align}\label{eqn:s-induc}
&mS_{(m)}q(S_{(0)})\nonumber\\
&=
-\sum_{i=0}^{N_q}q_i\sum_{j_0=1}^{m-1}
\sum_{\begin{subarray}{c}
0\leq j_1,\ldots,j_i\leq m-1\\
j_0+j_1+\cdots+j_i=m\end{subarray}}
j_0S_{(j_0)}S_{(j_1)}\cdots S_{(j_i)} \nonumber\\
&\quad{}+
\sum_{i=0}^{N_p}p_i
\sum_{\begin{subarray}{c}
0\leq j_1,\ldots,j_i\leq m-1\\
j_1+\cdots+j_i=m-1\end{subarray}}
S_{(j_1)}\cdots S_{(j_i)}.
\end{align}
Since $q_0\neq 0$ and $S_{(0)}$ is nilpotent,
$q(S_{(0)})$ is an invertible element of $\End_{\C}M$ and 
$q(S_{(0)})^{-1}$ is a polynomial in $S_{(0)}$.
Thus, it follows by induction on $m$ that 
every $S_{(m)}$ is a polynomial in $S_{(0)}$
and is uniquely determined by $D$ and $S_{(0)}$.
We conclude that $S(x)$ is uniquely determined by $D$ and $S_{(0)}$.

Since every $S_{(m)},m\in\Z$ is a polynomial in $S_{(0)}$
and $M$ is an indecomposable ${\mathcal A}_{M}(\C(s))$-module,
the nilpotent element $S_{(0)}$ conjugates to $J_n$ 
under the identification of $\End_{\C}M$ with $\Mat_{n}(\C)$.
Therefore, Lemma \ref{lemma:finite} implies that $S(x)$ and hence $(M,Y_M)$ with $\dim_{\C}M=n$
is uniquely determined up to isomorphism under the conditions in (1).

Conversely, suppose that $p(0)q(0)\neq 0$ and $\alpha=p(0)/q(0)$.
We shall construct $S(x)\in (\End_{\C}M)[[x]]$ which satisfies the conditions in (1).
We identify $\End_{\C}M$ with $\Mat_{n}(\C)$
by fixing  a basis of $M$ and we use Lemma \ref{lemma:finite}.
Set $S_{(0)}=J_n$.
By \eqref{eqn:s-induc} we can inductively define $S_{(m)}$ for $m=1,2,\ldots$.
Reversing the argument used to get \eqref{eqn:s-induc} above,
it is easy to see that
the obtained upper triangular matrix $S(x)=\sum_{m=0}^{\infty}S_{(m)}x^{m}\in (\Mat_{n}(\C))[[x]]$ 
satisfies \eqref{eqn:qsp}.  
Taking the semisimple part of \eqref{eqn:s-induc},
we obtain the formula by replacing $S_{(j)}$ with $S_{(j)}^{[0]}$ for all $j\in\Z$ in \eqref{eqn:s-induc}.
By this formula for $m=1$, we have $S_{(1)}^{[0]}=p_{0}/q_{0}=\alpha$.
By this formula again, an inductive argument on $m$  shows 
$S(x)^{[0]}=\sum_{m=0}^{\infty}S_{(m)}^{[0]}x^{m}$ is a non-constant element of $\C[[x]]E_n$. 
This implies that $S(x)-\beta$ is an invertible element of 
$(\Mat_{n}(\C))((x))$
for each $\beta\in\C$ 
and that $\ld(S(x)^{[0]})=1$ and $\lc(S(x)^{[0]})=\alpha$ since $S^{[0]}_{(0)}=0$.
Thus, $q(S(x))$ is an invertible element of 
$(\Mat_{n}(\C))((x))$ and hence \eqref{eqn:sqp} holds.
Since all coefficients of $S(x)$ are polynomials 
in $S_{(0)}=J_n$, 
$S_{(i)}S_{(j)}=S_{(j)}S_{(i)}$ for all $i,j\in\Z$.
By Lemma \ref{lemma:finite}, we have obtained  an $n$-dimensional vertex algebra $(\C(s),D)$-module $M$
with $\ld(S(x)^{(0)})=1$ and with $\lc(S(x)^{(0)})=\alpha$.
This completes the proof of (1).

\item
Next, suppose that $L=0$.
Setting $\tilde{s}=s-\alpha$, we have
\begin{align*}
D&=\dfrac{p(\tilde{s}+\alpha)}{q(\tilde{s}+\alpha)}\dfrac{d}{d\tilde{s}}.
\end{align*}
Thus, this case reduces to the case of $L>0$.
It follows by (1) that $S^{[0]}_{(1)}=p(\alpha)/q(\alpha)$.

\item
Finally, suppose that $L$ is a negative integer.
Setting $\tilde{s}=1/s$, we have
$Y_{M}(\tilde{s},x)^{[0]}=1/S(x)^{[0]}\in x\C((x))$ and 
\begin{align*}
D&=\dfrac{p(1/\tilde{s})}{q(1/\tilde{s})}(-\tilde{s}^{2})\dfrac{d}{d\tilde{s}}.
\end{align*}
It follows by $S(x)^{[0]}\neq 0$ that $S(x)$ is invertible in $(\End_{\C}M)((x)))$
and that $(S(x)^{-1})^{[0]}=(S(x)^{[0]})^{-1}$.
Since $S(x)^{-1}$ is a polynomial in $S(x)$,
all coefficients in $S(x)^{-1}$ are commutative.
Thus, this case also reduces to the case of $L>0$.
\end{enumerate}
\end{proof}
\end{theorem}

\section{Examples}
Throughout this section,
$D$ is a non-zero derivation of $\C(s)$ and
$f(s)=\sum_{j=0}^{m}f_js^j\in\C[s]$ is a square-free polynomial of degree $m$ at least $3$.
Set $K=\C(s)[t]/(t^2-f(s))$, which is a quadratic extension of $\C(s)$.
Let $\sigma$ be the generator of the Galois group of $K$ over $\C(s)$ mapping 
$t$ to $-t$.
Since $K$ is a finite Galois extension of $\C(s)$,
$D$ can be uniquely extends to 
a derivation of $K$, which is also denoted by $D$.

Let $(M,Y_M)$ be a finite-dimensional indecomposable $(\C(s),D)$-module.
In this section, we shall investigate untwisted or $\sigma$-twisted 
vertex algebra $(K,D)$-module structures over $(M,Y_M)$.
We denote $Y_{M}(s,x)$ by $S(x)$ and its semisimple part 
by  $S(x)^{[0]}=\sum_{i=L}^{\infty}S^{[0]}_{(i)}x^i$ with
$S^{[0]}_{(L)}\neq 0$ as in Section 4.
Theorem \ref{theorem:untwist} says that $L=\ld(S(x)^{[0]})=1,0,$ or $-1$. 
\begin{proposition}
Let $M$ be a finite-dimensional indecomposable 
vertex algebra $(\C(s),D)$-module
and let $(M,\tilde{Y}_{M})$ be a $g$-twisted vertex algebra $(K,D)$-module structure 
over $(M,Y_M)$ in Theorem \ref{theorem:correspondence}, where $g=1$ or $\sigma$.
Then
\begin{enumerate}
\item 
In the case of $\ld(S(x)^{[0]})=1$, 
$g=\sigma$ if and only if $f_0=0$.
\item
In the case of $\ld(S(x)^{[0]})=0$, 
$g=\sigma$ if and only if $f(S^{[0]}_{(0)})=0$.
\item
In the case of $\ld(S(x)^{[0]})=-1$, 
$g=\sigma$ if and only if $\deg f$ is odd.
\end{enumerate}
\end{proposition}
\begin{proof}
We denote by $\sqrt{f(S(x)^{[0]})}$ a square root of $f(S(x)^{[0]})$ in 
$\Omega=\cup_{i=1}^{\infty}\C((x^{1/i}))$.
We denote by $\psi$ the $\C$-algebra homomorphism $\psi[\C(s),(M,Y_{M})]\colon \C(s)\rightarrow
\C((x))$ defined just after \eqref{eq:proj} and by $Q$ 
the image of $\psi$.
Let $\hat{K}$ denote 
the splitting field for $Z^2-\psi(f(s))=Z^2-f(S(x)^{[0]})\in Q[Z]$ in $\Omega$
as in the proof of Theorem \ref{theorem:correspondence}.
We have $\hat{K}\cap \C((x))=\hat{K}$ if and only if 
$\sqrt{f(S(x)^{[0]})}\in\C((x))$.
By the argument just after \eqref{eq:galois},
$g$ is a generator of $\Gal(\hat{K}/\hat{K}\cap \C((x)))$.
Thus, $g=1$ if and only if $\sqrt{f(S(x)^{[0]})}\in \C((x))$.
We use the following expansion of $f(S(x)^{[0]})$:
\begin{align}
&f(S(x)^{[0]})=\sum_{j=0}^{m}f_j(\sum_{i=L}^{\infty}S^{[0]}_{(i)}x^i)^j\nonumber\\
&=\left\{
\begin{array}{ll}
f_0+f_1S^{[0]}_{(1)}x+\cdots,&\mbox{if }L=1,\\
f(S^{[0]}_{(0)})+S^{[0]}_{(1)}f^{\prime}(S^{[0]}_{(0)})x+\cdots,&\mbox{if }L=0,\\
f_{m}(S^{[0]}_{(-1)})^{m}x^{-m}
+\cdots
,&\mbox{if }L=-1.
\end{array}
\right.\label{eqn:expand}
\end{align}

\begin{enumerate}
\item In this case if $f_0\neq 0$, then $\sqrt{f(S(x)^{[0]})}\in \C((x))$.
If $f_0=0$, then $f_1\neq 0$ since $f$ is square-free.
Thus, $\sqrt{f(S(x)^{[0]})}\not\in \C((x))$.
\item In this case $S^{[0]}_{(1)}\neq 0$ by Theorem \ref{theorem:untwist} (2).
If $f(S^{[0]}_{(0)})\neq 0$, then $\sqrt{f(S(x)^{[0]})}\in \C((x))$.
If $f(S^{[0]}_{(0)})=0$, then $\sqrt{f(S(x)^{[0]})}\not\in \C((x))$
since $f$ is square-free and $S^{[0]}_{(1)}\neq 0$.
\item In this case, the assertion follows easily from (\ref{eqn:expand}). 
\end{enumerate}
\end{proof}

\section*{Acknowledgements}
The author thanks the referee for pointing out a mistake 
in the proof of the earlier version of Theorem \ref{theorem:correspondence}.

\end{document}